\documentclass[a4paper]{amsart}
  
\usepackage{amsthm,amsfonts,amsmath,amssymb}

\usepackage{isomath}
\usepackage{lmodern} 
\usepackage{setspace}
\usepackage{graphicx,enumerate}  
\usepackage{verbatim} 
\usepackage[capitalise]{cleveref}
\newtheorem{theorem}{Theorem}
\newtheorem{corollary}[theorem]{Corollary}
\newtheorem{lemma}[theorem]{Lemma}
\newtheorem{proposition}[theorem]{Proposition}
\newtheorem{problem}[theorem]{Problem}

\newenvironment{claimproof}{\noindent\textit{Proof.}}{\hfill$\square$}

\theoremstyle{definition}

\theoremstyle{remark}
\newtheorem{remark}[theorem]{Remark}

\crefname{remark}{Remark}{Remarks}

\theoremstyle{definition}
\newtheorem{claim}{Claim}
\usepackage{etoolbox}
\AtEndEnvironment{proof}{\setcounter{claim}{0}}

\theoremstyle{remark}

\usepackage{etoolbox}
\AtEndEnvironment{proof}{\setcounter{case}{0}}
\AtEndEnvironment{claimproof}{\setcounter{case}{0}}

\usepackage{mathtools}

\DeclarePairedDelimiter\floor{\lfloor}{\rfloor}

\DeclareMathOperator{\conv}{conv}

\DeclareMathOperator{\st}{star}
\DeclareMathOperator{\a-st}{astar}

\newcommand{\D}{\mathcal D}
\newcommand{\R}{\mathbb{R}}

\newcommand{\C}{\mathcal{C}}
\newcommand{\B}{\mathcal{B}}

\newcommand{\St}{\mathcal{S}}

\renewcommand{\vec}[1]{\boldsymbol{{#1}}}
 
\crefname{claim}{Claim}{Claims}
\crefname{problem}{Problem}{Problems}
\title{Connectivity of cubical polytopes}
      
\usepackage{amsaddr}

\author{Hoa T. Bui}
\address{Centre for Informatics and Applied Optimisation, Federation University Australia}
\email{\texttt{h.bui@federation.edu.au}}
\author{Guillermo Pineda-Villavicencio \& Julien Ugon}
\address{Centre for Informatics and Applied Optimisation, Federation University Australia\\School of Information Technology, Deakin University}
\email{\texttt{julien.ugon@deakin.edu.au}} 
\email{\texttt{work@guillermo.com.au}}

\thanks{Hoa T. Bui is supported by an Australian Government Research Training Program (RTP) Stipend and RTP Fee-Offset Scholarship through Federation University Australia. Julien Ugon's research was partially supported by ARC discovery project DP180100602.}

\begin{document}
 
\begin{abstract} A cubical polytope is a polytope with all its facets being combinatorially equivalent to cubes. We deal with the connectivity of the graphs of cubical polytopes. We first establish that, 
for any $d\ge 3$, the graph of a cubical $d$-polytope with minimum degree $\delta$  is $\min\{\delta,2d-2\}$-connected. Second, we show, for any $d\ge 4$,  that every minimum separator of cardinality at most $2d-3$  in such a graph  consists of  all the neighbours of some vertex and that removing the vertices of the separator from the graph leaves exactly two components, with one of them being the vertex itself.
\end{abstract} 
 
\maketitle       
\section{Introduction} 
The $k$-dimensional {\it skeleton} of a polytope $P$ is the set of all its faces of dimension of at most $k$. The 1-skeleton of $P$ is the {\it graph} $G(P)$ of $P$. We denote by $V(P)$ the vertex set of $P$.

This paper studies the (vertex) connectivity of a cubical polytope, the (vertex) connectivity of the graph of the polytope. A {\it cubical} $d$-polytope is a polytope with all its facets being cubes.  By a cube we mean any polytope that is combinatorially equivalent to a cube; that is, one whose face lattice  is isomorphic to the face lattice of a cube.

Unless otherwise stated, the graph theoretical notation and terminology follow from \cite{Die05} and the polytope theoretical notation and terminology from \cite{Zie95}. Moreover, when referring to graph-theoretical properties of a polytope such as degree and connectivity, we mean properties of its graph.

In  the three-dimensional world, Euler's formula implies that the graph of a cubical 3-polytope $P$ has $2|V(P)|-4$ edges, and hence its minimum degree is three; the {\it  degree}  of a vertex records the number of edges incident with the vertex. Besides, every 3-polytope is $3$-connected by Balinski's theorem \cite{Bal61}. Hence the dimension, minimum degree and connectivity of a cubical 3-polytope all coincide.

\begin{theorem}[Balinski {\cite{Bal61}}]\label{thm:Balinski} The graph of a $d$-polytope is $d$-connected. 
\end{theorem}

This equality between dimension, minimum degree and connectivity of a cubical polytope no longer holds in higher dimensions. In Blind and Blind's classification of cubical $d$-polytopes where every vertex has  degree $d$ or $d+1$  \cite{BliBli98}, the authors exhibited cubical $d$-polytopes with the same graph as the $(d+1)$-cube; for an explicit example, check \cite[Sec.~4]{JosZie00}. More generally, the paper  \cite[Sec.~6]{JosZie00} exhibited cubical $d$-polytopes with the same $(\floor{d/2}-1)$-skeleton as the $d'$-cube for every $d'>d$, the so-called {\it neighbourly cubical $d$-polytopes}. And even more generally, Sanyal and Ziegler \cite[p.~422]{SanZie10}, and later Adin, Kalmanovich and Nevo \cite[Sec.~5]{AdiKalNev18}, produced cubical $d$-polytopes with the same $k$-skeleton as the $d'$-cube for every $1\le k\le \floor{d/2}-1$ and every $d'>d$,  the so-called {\it $k$-neighbourly cubical $d$-polytopes}. Thus the minimum degree or  connectivity of a cubical $d$-polytope for $d\ge 4$ does not necessarily coincide with its dimension; this is what one would expect. However, somewhat surprisingly, we can prove a result connecting the connectivity of a cubical polytope to its minimum degree, regardless of the dimension; this is a vast generalisation of a similar, and well-known, result in the $d$-cube \cite[Prop.~1]{Ram04}; see also \cref{sec:cube-connectivity}.   
      
\vspace{0.2cm}

Define a {\it separator} of a polytope as a set of vertices disconnecting the graph of the polytope. Let  $X$ be a set of vertices in a graph $G$.  Denote by $G[X]$ the subgraph of $G$ induced by $X$, the subgraph of $G$ that contains all the edges of $G$ with vertices in $X$.   Write $G-X$ for $G[V(G)\setminus X]$;  that is, the subgraph $G-X$ is obtained by removing the vertices in $X$ and their incident edges. Our main result is the following.

\vspace{0.2cm} 

\noindent {\bf Theorem} (Connectivity Theorem). {\it A cubical $d$-polytope $P$ with minimum degree $\delta$ is $\min\{\delta,2d-2\}$-connected for every $d\ge 3$. 

Furthermore, for any $d\ge 4$, every minimum separator $X$ of cardinality at most $2d-3$ consists of  all the neighbours of some vertex, and the subgraph $G(P)-X$ contains  exactly two components, with one of them being the vertex itself.}

 A {\it simple} vertex in a $d$-polytope is a vertex of degree $d$; otherwise we say that the vertex is {\it nonsimple}.  An immediate corollary of the theorem is the following.
 
\vspace{0.2cm}

\noindent {\bf Corollary.} {\it A cubical $d$-polytope with no simple vertices is $(d+1)$-connected.}
 
\begin{remark}
The connectivity theorem is best possible in the sense that there are infinitely many cubical $3$-polytopes with minimum separators not consisting of the neighbours of some vertex (\cref{fig:cubical-3-polytopes}). 
\end{remark} 
    
It is not hard to produce examples of polytopes with differing values of minimum degree and connectivity. The {\it connected sum} $P_{1}\#P_{2}$ of two $d$-polytopes $P_{1}$ and $P_{2}$ with  two projectively isomorphic facets $F_{1}\subset P_{1}$ and $F_{2}\subset P_{2}$ is obtained by gluing $P_{1}$ and $P_{2}$ along $F_{1}$ and $F_{2}$ \cite[Example~8.41]{Zie95}. Projective transformations on the polytopes $P_{1}$ and $P_{2}$, such as those in \cite[Def.~3.2.3]{Ric06}, may be required for $P_{1}\#P_{2}$ to be convex. \Cref{fig:connected-sum} depicts this operation. A connected sum of two copies of a cyclic $d$-polytope with $d\ge 4$ and $n\ge d+1$ vertices  (\cite[Thm.~0.7]{Zie95}), which is a polytope whose facets are all simplices, results in a $d$-polytope of minimum degree $n-1$ that is $d$-connected but not $(d+1)$-connected. 

\begin{figure}
\includegraphics{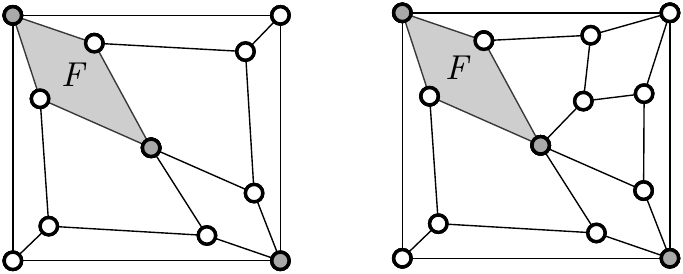}
\caption{Cubical $3$-polytopes with minimum separators not consisting of the neighbours of some vertex. The vertices of the separator are coloured in gray. The removal of the vertices of a face $F$ does not leave a 2-connected subgraph: the remaining vertex in gray disconnects the subgraph. Infinitely many more examples can be generated by using well-known expansion operations such as those in \cite[Fig.~3]{BriGreGre05}.}\label{fig:cubical-3-polytopes}
\end{figure}

\begin{figure}
\includegraphics{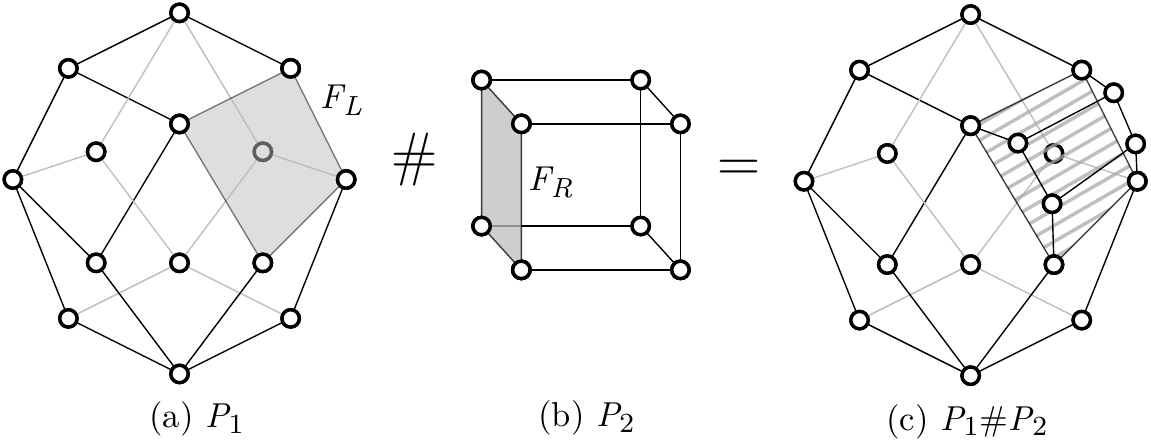}
\caption{Connected sum of two cubical polytopes.}\label{fig:connected-sum} 
\end{figure} 

On our way to prove the connectivity theorem we prove  results of independent interest, for instance, the following (\cref{cor:Removing-facet-(d-2)-connectivity} in \cref{sec:cubical-connectivity}). 

\vspace{0.2cm}

\noindent {\bf Corollary. }{\it   Let $P$ be a cubical $d$-polytope and let $F$ be a proper face of $P$. Then the subgraph  $G(P)-V(F)$ is $(d-2)$-connected.} 

\vspace{0.2cm}

\begin{remark}
The examples of \cref{fig:cubical-3-polytopes} also establish that the previous corollary is best possible in the sense that the removal of the vertices of a proper face $F$ of a cubical $d$-polytope does not always leave a $(d-1)$-connected subgraph of the graph of the polytope.
\end{remark}  

The corollary gives another unusual property of cubical polytopes. A tight result of Perles and Prabhu \cite[Thm.~1]{PerPra93} implies that the removal of the vertices of any $k$-face ($-1\le k\le d-1$) from a $d$-polytope leaves a $\max\{1,d-k-1\}$-connected subgraph of the graph of the polytope. 
 
The connectivity theorem also gives rise to the following corollary and  open problem. 
\vspace{0.2cm}

\noindent {\bf Corollary.} {\it There are functions $f:\mathbb{N}\rightarrow \mathbb{N}$ and $g:\mathbb{N}\rightarrow \mathbb{N}$ such that, for every $d\ge 4$, 
\begin{enumerate}[(i)]
\item the function $f(d)$ gives the maximum number such that every cubical $d$-polytope with minimum degree $\delta\le f(d)$ is $\delta$-connected; 
\item the function $g(d)$ gives the maximum number such that every minimum separator with cardinality at most $g(d)$ of every cubical $d$-polytope consists of the neighbourhood of some vertex; and
\item $2d-3\le g(d)$ and $g(d)< f(d)$.
\end{enumerate}}

An exponential bound in $d$ for $f(d)$ is readily available. The connected sum of two copies of a neighbourly cubical $d$-polytope with minimum degree $\delta> 2^{d-1}$, which exists by \cite[Thm.~16]{JosZie00}, results in a cubical $d$-polytope with minimum degree $\delta$ and with a minimum separator of cardinality $2^{d-1}$, the number of vertices of the facet along which we glued. The cardinality of this separator gives at once the announced upper bound. This exponential bound in conjunction with  the connectivity theorem gives that 
\begin{equation}\label{eq:bounds}
	2d-3\le g(d)< f(d)\le 2^{d-1}.
\end{equation} 
The following problem naturally arises.
  
\begin{problem}\label{prob:bounds}
For $d\ge 4$ provide precise values for the functions $f(d)$ and $g(d)$ or improve the lower and upper bounds in \eqref{eq:bounds}. 
\end{problem}
We suspect that both functions are linear in $d$.

Using ideas developed in this paper, we studied further connectivity properties of cubical $d$-polytopes \cite{ThiPinUgo18}. A graph is {\it $k$-linked} if, for every set of $2k$ distinct vertices $\{s_{1},\dots,s_{k},t_{1},\ldots,t_{k}\}$, there exist disjoint paths $L_{1},\ldots, L_{k}$ such that the endpoints of $L_{i}$ are $s_{i}$ and $t_{i}$.  We proved that the graph of a cubical $d$-polytope is $\floor{(d+1)/2}$-linked for $d\ne 3$. 
 
\section{Preliminary results}
\label{sec:preliminary}
This section groups a number of  results that will be used in later sections of the paper. 
 
The definitions of polytopal complex and strongly connected complex play an important role in the paper. A {\it polytopal complex} $\C$ is a finite nonempty collection of polytopes in $\R^{d}$ where the faces  of each polytope in $\C$ all belong to $\C$ and where polytopes intersect only at faces (if $P_{1}\in \C$ and $P_{2}\in \C$ then $P_{1}\cap P_{2}$ is a face of both $P_{1}$ and $P_{2}$). The empty polytope is always in $\C$. The {\it dimension} of a complex $\C$ is the largest dimension of a polytope in $\C$;  if $\C$ has dimension $d$ we say that $C$ is a {\it $d$-complex}. Faces of a complex of largest and second largest dimension are called {\it facets} and {\it ridges}, respectively. If each of the faces of a complex $\C$ is contained in some facet we say that $\C$ is {\it pure}. 

Given  a polytopal complex $\C$ with vertex set $V$ and a subset $X$ of $V$,  the subcomplex of $\C$ formed by all the faces of $\C$ containing only vertices from $X$ is called {\it induced} and is denoted by $\C[X]$.  Removing from $\C$  all the vertices in a subset $X\subset V(\C)$ results in the subcomplex $\C[V(\C)\setminus X]$, which we write as $\C-X$. We say that a subcomplex $\C'$ is a {\it spanning} subcomplex of $\C$ if $V(\C')=V(\C)$. The {\it graph} of a complex is the undirected graph formed by the vertices and edges of the complex. As in the case of polytopes, we denote the graph of a complex $\C$ by $G(\C)$.  A pure polytopal complex $\C$ is {\it strongly connected} if every pair of facets $F$ and $F'$ is connected by a path $F_{1}\ldots F_{n}$ of facets in $\C$ such that $F_{i}\cap F_{i+1}$ is a ridge of $\C$, $F_{1}=F$, and $F_{n}=F'$; we say that such a path is a {\it $(d-1,d-2)$-path} or a {\it facet-ridge path} if the dimensions of the faces can be deduced from the context. From the definition, it follows that every 0-complex is trivially strongly connected and that every complex contains a spanning 0-subcomplex.  
 
The relevance of strongly connected complexes stems from the ensuing result of Sallee.

\begin{proposition}[{\cite[Sec.~2]{Sal67}}]\label{prop:connected-complex-connectivity} The graph of a strongly connected $d$-complex is $d$-connected. 
\end{proposition}

Strongly connected complexes can be defined from a $d$-polytope $P$. Two basic examples are given by the complex of all faces of $P$, called the {\it complex} of $P$ and denoted by $\C(P)$, and the complex of all proper faces of $P$, called the {\it boundary complex} of $P$ and denoted by $\B(P)$.   For a polytopal complex $\C$, the {\it star} of a face $F$ of $\C$, denoted $\st(F,\C)$, is the subcomplex of $\C$ formed by all the faces containing $F$, and their faces; the {\it antistar} of a face $F$ of $\C$, denoted $\a-st(F,\C)$, is the subcomplex of $\C$ formed by all the faces disjoint from $F$. That is, $\a-st(F,\C)=\C-V(F)$. Unless otherwise stated, when defining stars and antistars in a polytope, we always assume the underlying complex is the boundary complex of the polytope.    
  
Some complexes defined from a $d$-polytope are strongly connected $(d-1)$-complexes, as the next proposition attests; the parts about the boundary complex and the antistar of a vertex already appeared in \cite{Sal67}. 

\begin{proposition}[{\cite[Cor.~11, Thm.~3.5]{Sal67}}]\label{prop:st-ast-connected-complexes} Let $P$ be a $d$-polytope. Then, the boundary complex $\B(P)$ of $P$, and the star and antistar of a vertex in $\B(P)$, are all strongly connected $(d-1)$-complexes of $P$.
\end{proposition}
\begin{proof} Let $\psi$ define the natural anti-isomorphism from the face lattice of $P$ to the face lattice of its dual $P^{*}$. 

The three complexes are pure. The complex $\B(P)$ is clearly pure, and so is the star of a vertex. Perhaps a sentence may be appropriate for the antistar of a vertex: a face of $P$ that does not contain the vertex must lie in a facet that does not contain the vertex.   We proceed to prove the  strong connectivity of the complexes.
   
The statement about $\B(P)$ was already proved in \cite[Cor.~2.11]{Sal67}. The facets in $\B(P)$ correspond to vertices in $P^{*}$. The existence of a facet-ridge path in $\B(P)$ between any two facets $F_{1}$ and $F_{2}$ of $\B(P)$ amounts to the existence of a vertex-edge path in $P^{*}$ between the vertices $\psi(F_{1})$ and $\psi(F_{2})$ of $P^{*}$. That $\B(P)$ is a strongly connected $(d-1)$-complex now follows from the connectivity of the graph of $P^{*}$ (Balinski's theorem).
 
The assertion about the star of a vertex does not seem to explicitly appear in \cite{Sal67}. The facets in the star $\St$ of a vertex $v$ in $\B(P)$ correspond to the vertices in the facet $\psi(v)$ in $P^{*}$. The existence of a facet-ridge path in $\St$ between any two facets $F_{1}$ and $F_{2}$ of $\St$ amounts to the existence of a vertex-edge path in $\psi(v)$ between the vertices $\psi(F_{1})$ and $\psi(F_{2})$ of $\psi(v)$. That $\St$ is a strongly connected $(d-1)$-complex follows from the connectivity of the graph of $\psi({v})$ (Balinski's theorem).

The assertion about the antistar  of a vertex $v$ was first shown  in \cite[Thm.~3.5]{Sal67}. The facets in $\a-st(v)$ correspond to the vertices of $P^{*}$ that are not in $\psi(v)$. That is, if $F_1$ and $F_{2}$ are any two facets of $\a-st(v)$, then  $\psi(F_{1}),\psi(F_{2})\in V(P^{*})\setminus V(\psi(v))$. The existence of a facet-ridge path between $F_1$ and $F_{2}$ in $\a-st(v)$  amounts to the existence of a vertex-edge path between $\psi(F_{1})$ and $\psi(F_{2})$ in the subgraph $G(P^{*})- V(\psi(v))$ of $G(P^{*})$. The removal of the vertices of a facet does not disconnect the graph of a polytope \cite[Thm.~3.1]{Sal67}, wherefrom it follows that $G(P^{*})- V(\psi(v))$ is connected, as desired.
\end{proof}

\section{Connectivity of the $d$-cube}   
\label{sec:cube-connectivity}
We unveil some further properties of the cube, whose  proofs exploit the realisation of a $d$-cube as a $0-1$ $d$-polytope \cite{Zie00}. A {\it $0-1$ $d$-polytope} is a $d$-polytope whose vertices have coordinates in  $\{0,1\}^{d}$. Here $\{0,1\}^{d}$ denotes the set of all $d$-element sequences from $\{0,1\}$. 
  
We next give some basic properties of the $d$-cube, including some specific to its realisation as a $0-1$ polytope. 

\begin{remark}[Basic properties of the $d$-cube]\label{rmk:cube-properties} Let $\vec x=(x_{1},\ldots,x_{d})$ with $x_{i}\in\{0,1\}$ be a vertex of the $0-1$ $d$-cube $Q_{d}$. 
\begin{enumerate}[(i)]
\item Every two facets of $Q_{d}$ either intersect at a ridge or are disjoint.
\item Each of the $2d$ facets of $Q_{d}$ is the convex hull of a set of the form 
\begin{align*}
F_{i}^{0}:=\conv \{\vec x\in V(Q_{d}): x_{i}=0\}\;\text{or}\; F_{i}^{1}:=\conv \{\vec x\in V(Q_{d}): x_{i}=1\},
\end{align*} 
for $i$ in $[1,d]$, the interval $1,\ldots,d$.

\item A $(d-k)$-face is the intersection of exactly $k$ facets, and thus, its vertices have the form \[\{\vec x \in V(Q_{d}): x_{i_{1}}=0,\ldots,x_{i_{r}}=0,x_{i_{r+1}}=1,\ldots,x_{i_{k}}=1\}\] for $k\in [1,d]$ and $r\in[0,k]$.
\end{enumerate}
	  
\end{remark}

While it is true that the antistar of a vertex in a $d$-polytope is always a strongly connected $(d-1)$-complex (\cref{prop:st-ast-connected-complexes}), it is far from true that this extends to higher dimensional faces. Consider any $d$-polytope $P$ with a simplex facet $J$ that contains at least one vertex $v$ of degree $d$ in $P$. Let  $F$ be any face in $J$ that does not contain $v$. Then the vertex $v$ has degree $d-|V(F)|$ in  the subcomplex $P- V(F)$. Since every vertex in a pure $(d-1)$-complex has degree at least $d-1$, the antistar of $F$ in $\B(P)$, which contains $v$, cannot be a pure $(d-1)$-complex for $\dim F\ge 1$. This extension is however possible for the $d$-cube.

\begin{lemma}
\label{lem:cube-face-complex} Let $F$ be a proper face in the $d$-cube $Q_{d}$. Then the antistar of $F$ is a strongly connected $(d-1)$-complex. 
\end{lemma}
\begin{proof} Without loss of generality, assume that  $Q_{d}$  is given as a $0-1$ polytope, and for the sake of concreteness, that our proper face $F$ is defined as $\conv \{\vec x\in V(Q_{d}): x_{1}=0,\ldots,x_{k}=0\}$ (\cref{rmk:cube-properties}(iii)). That is, $F=F_{1}^{0}\cap \cdots\cap F_{k}^{0}$; refer to \cref{rmk:cube-properties}(ii).

We claim that the antistar of $F$ is the pure $(d-1)$-complex \[\C:=\C(F_{1}^{1})\cup \cdots\cup \C(F_{k}^{1});\] refer to \cref{rmk:cube-properties}(ii)-(iii).

We proceed by proving that $\a-st(F,Q_{d})\subseteq \C$. Take any $(d-l)$-face $K\not \in \C$. Then $K=J_{1}\cap \cdots \cap J_{l}$ for some facets $J_{i}$ of $Q_{d}$. A facet $J_{i}$ is defined by either $\conv \{\vec x:V(Q_{d}):x_{j}=1\; \text{for some $j\in [k+1,d]$}\}$ or $\conv \{\vec x:V(Q_{d}):x_{j}=0\; \text{for some $j\in [1,d]$}\}$.  According to \cref{rmk:cube-properties}(iii), for $l\in [1,d]$ and $r\in[0,l]$, we get that \[K=\conv \{\vec x\in V(Q_{d}):x_{i_{1}}=0,\ldots,x_{i_{r}}=0,x_{i_{r+1}}=1,\ldots,x_{i_{l}}=1\}.\] From the form of the facets $J_{i}$ it follows that  $i_{j}\ge k+1$ for all $j\in [r+1, l]$. Hence there is a vertex $\vec x=(x_{1},\ldots,x_{d})$ in $K$ satisfying $x_{1}=\cdots =x_{k}=0$, which implies that $\{\vec x\}\subset K\cap F$. That is, $K\not \in \a-st(F,Q_{d})$. 

To prove that $\C \subseteq \a-st(F,Q_{d})$ holds, observe that, if $K\in \C$ then it is in a facet $F_{i}^{1}$ for some $i\in[1,k]$, and therefore, it belongs to $\a-st(F,Q_{d})$. Hence $\C=\a-st(F,Q_{d})$.

That $\C$ is strongly connected follows from noting that the facets $F_{1}^{1}, \ldots, F_{k}^{1}$ are pairwise nondisjoint, and therefore, pairwise intersect at $(d-2)$-faces (\cref{rmk:cube-properties}(i)). 
\end{proof}

 \cref{prop:known-cube-cutsets} is well known \cite[Prop.~1]{Ram04}, but we are not aware of a reference for \cref{prop:cube-cutsets}. 

\begin{proposition}[{\cite[Prop.~1]{Ram04}}]\label{prop:known-cube-cutsets}
Any separator $X$ of cardinality $d$ in $Q_{d}$ consists of the $d$ neighbours of some vertex in the cube, and the subgraph $G(Q_{d})-X$ has exactly two components, with one of them being the vertex itself.
\end{proposition}
\begin{proof}
A proof can be found in \cite[Prop.~1]{Ram04}: essentially, one proceeds by induction on $d$, considering the effect of the separator on a pair of disjoint facets.  
\end{proof}

\begin{proposition}\label{prop:cube-cutsets}  Let $y$ be a vertex of the $d$-cube $Q_{d}$  and let   $Y$ be a subset of the neighbours of $y$ in $Q_{d}$. Then the  subcomplex of $Q_{d}$ induced by  $V(Q_{d})\setminus (\{y\}\cup Y)$  contains a spanning strongly connected $(d-2)$-subcomplex.
\end{proposition} 

\begin{proof} Without loss of generality, assume that  $Q_{d}$  is given as a $0-1$ polytope, and for the sake of concreteness, that $y=(0,\ldots,0)$ and $Y=\{\vec e_{1},\ldots,\vec e_{k}\}$ where $\vec e_{i}$ denotes the standard unit vector with the $i$-entry equal to one.

Let $\C:=Q_{d}- (\{y\}\cup Y)$, the subcomplex of $Q_{d}$ induced by  $V(Q_{d})\setminus (\{y\}\cup Y)$. Consider the $d-k$ ridges
\[R_{i}:=\conv \{\vec x\in V(Q_{d}): x_{1}=0,x_{i}=1\}\;\text{for $i\in [k+1,d]$}\] and the ${d\choose 2}$ ridges \[R_{i,j}:=\conv  \{\vec x\in V(Q_{d}): x_{i}=1,x_{j}=1\} \text{for some $i,j\in[1,d]$ with $i\ne j$.}\] 
Let $\C':=\C(R_{k+1})\cup\cdots\cup \C(R_{d})\cup \C(R_{1,2})\cup \cdots\cup \C(R_{d-1,d})$. Then $\C'$ is a pure $(d-2)$-subcomplex of $\C$. 
 
We show that $\C'$ is a spanning subcomplex of $\C$. Let $\vec x=(x_{1},\ldots,x_{d})$ be a vertex in $V(Q_{d})\setminus (\{y\}\cup Y)$. Then either $\vec x=\vec e_{i}$ for some $i=k+1,\ldots, d$ or $x_{i}=x_{j}=1$ for some $i,j\in[1,d]$ with $i\ne j$; see \cref{rmk:cube-properties}(iii).
In the former case, the vertex $\vec x$ lies in the $(d-2)$-face $R_{i}$, and in the latter case, the vertex $\vec x$ lies in the $(d-2)$-face $R_{i,j}$.
Therefore $\vec x\in \C'$. We next show that $\C'$ is strongly connected.  

Take any two distinct ridges \(R\) and \(R'\) from $\C'$. We consider three cases based on the form of $R$ and $R'$.

Suppose that \(R=R_{i}\)  and \(R'=R_{j}\) for $i,j\in[k+1,d]$ and $i\ne j$. Then there is a \((d-2,d-3)\)-path $L$ of length one from \(R\) to \(R'\) through their common \((d-3)\)-face \(\conv \{\vec x\in V(Q_{d}):x_1=0,x_i=1,x_j=1\}\). That is, $L:=RR'$.

Next suppose that \(R=R_{i}\) and \(R'=R_{j,l}\) for $j,l\in[1,d]$ with $j\ne l$.  If \(i=j\)  there is a \((d-2,d-3)\)-path of length one from \(R\) to \(R'\) through the common \((d-3)\)-face \(\conv \{\vec x\in V(Q_{d}):x_1=0,x_i=1,x_l=1\}\). If \(i\neq j\), and consequently \(i\neq l\), then there is a \((d-2,d-3)\)-path $L$ of length two from \(R\) to \(R'\) through the \((d-2)\) face $R_{i,l}$, which shares the \((d-3)\)-face $\conv \{\vec x\in V(Q_{d}):x_{1}=0,x_{i}=1,x_{l}=1\}$ with \(R\) and the \((d-3)\)-face $\conv \{\vec x\in V(Q_{d}):x_{i}=1,x_{j}=1,x_{l}=1\}$ with \(R'\). That is, $L:=RR_{i,l}R'$.

Finally suppose that \(R=R_{i,j}\) with $i\ne j$  and \(R'=R_{l,m}\) with $l\ne m$. If \(i=l\)  there is a \((d-2,d-3)\)-path from \(R\) to \(R'\) through the common \((d-3)\)-face \(\conv \{\vec x\in V(Q_{d}):x_i=1,x_j=1,x_m=1\}\). If \(\{i,j\}\cap\{l,m\}=\emptyset\)  then there is a \((d-2,d-3)\)-path $L$ of length two from \(R\) to \(R'\) through the \((d-2)\)-face $R_{i,l}$, which shares  the \((d-3)\)-face $\conv \{\vec x\in V(Q_{d}):x_{i}=1,x_{j}=1,x_{l}=1\}$ with \(R\) and the \((d-3)\)-face $\conv \{\vec x\in V(Q_{d}):x_{i}=1,x_{l}=1,x_{m}=1\}$ with \(R'\). That is, $L:=RR_{i,l}R'$.\end{proof}

\begin{remark}\label{rmk:cubical-neighbours}
In \cref{prop:cube-cutsets}, the subcomplex of $Q_{d}$ induced by $V(Q_{d})\setminus (\{y\}\cup Y)$, in the proof of \cref{prop:cube-cutsets} denoted by $\C$, is pure if and only if $Y$ is the set of all neighbours of $y$. Let $Y=\{\vec e_{1},\ldots,\vec e_{k}\}$ and $y=(0,\ldots,0)$. If $k<d$ then the facets $\conv\{\vec x\in V(Q_{d}):x_{\ell}=1\}$ for $\ell\in[k+1,d]$ are in $\C$ and the ridge $\conv\{\vec x\in V(Q_{d}):x_{i}=1,x_{j}=1\}$ for $i,j\in[1,k]$ and $i\ne j$ is in $\C$ but the facets $\conv\{\vec x\in V(Q_{d}):x_{i}=1\}$ and $\conv\{\vec x\in V(Q_{d}):x_{j}=1\}$ are not in $\C$. Thus $\C$ is nonpure. If instead $k=d$ then no facet is in $\C$, and the vector coordinates of every vertex in $\C$ has at least two entries with ones, and thus, it is contained in some ridge $\conv\{\vec x\in V(Q_{d}):x_{i}=1,x_{j}=1\}$ for $i,j\in [1,d]$ and $i\ne j$, which is in $\C$. Thus $\C$ is a pure $(d-2)$-subcomplex of $Q_{d}$, and it coincides with the complex $\C'$. Figure \ref{fig:prop-cube-cutsets} illustrates \cref{prop:cube-cutsets}.
\end{remark}

\begin{figure}  
\includegraphics{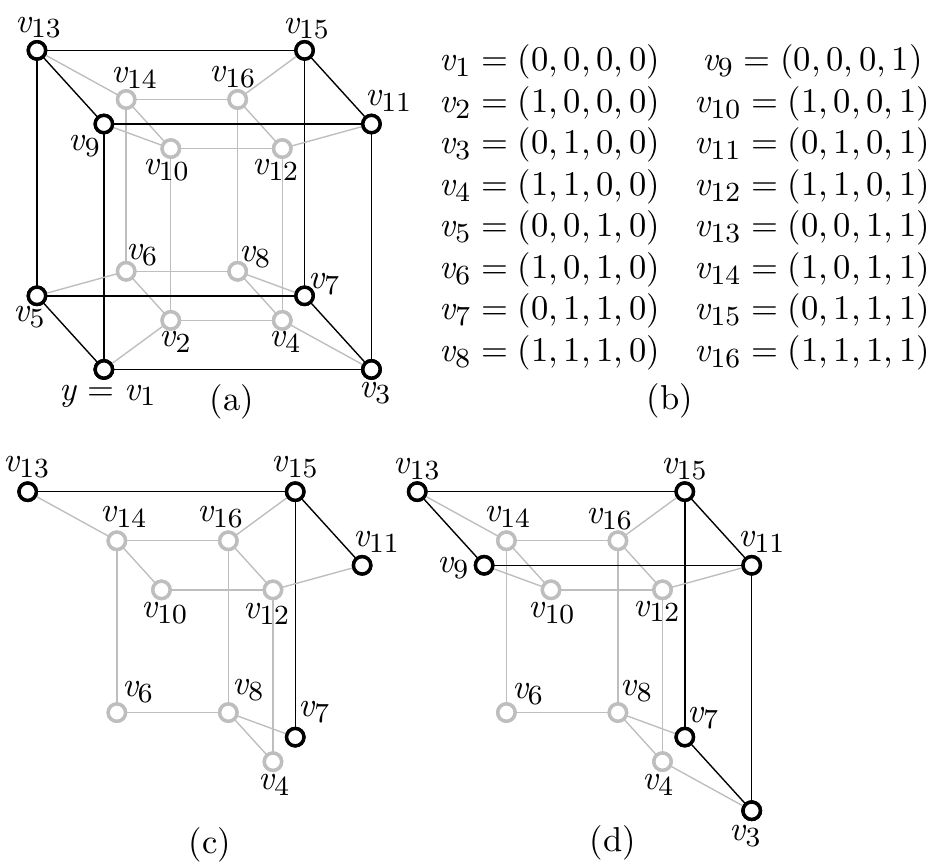}
\caption{Complexes in the 4-cube. (a) The 4-cube with the vertex $y=(0,0,0,0)$ singled out. The vertex labelling corresponds to a realisation of the 4-cube as a $0-1$ polytope. (b)  Vertex coordinates as elements of $\{0,1\}^{4}$. (c) The strongly connected 2-complex $\C$ induced by $V(Q_{4})\setminus (\{y\}\cup Y)$ where $Y=\{v_{2},v_{3},v_{5},v_{9}\}$. Every face of $\C$ is contained in a 2-face of the cube. (d) The nonpure complex $\C$ induced  by $V(Q_{4})\setminus (\{y\}\cup Y)$ where $Y=\{v_{2}=\vec e_{1},v_{5}=\vec e_{3}\}$. The 2-face $\conv \{v_{6},v_{8},v_{14},v_{16}\}=\conv\{\vec x\in V(Q_{4}):x_{1}=1,x_{3}=1\}$ of $\C$ is not contained in any 3-face, and there are two 3-faces in $\C$, namely $\conv \{v_{3},v_{4},v_{7},v_{8}, v_{11},v_{12},v_{15},v_{16}\}=\conv\{\vec x\in V(Q_{4}):x_{2}=1\}$ and $\conv \{v_{9},v_{10},v_{11},v_{12},v_{13},v_{14}, v_{15},v_{16}\}=\conv\{\vec x\in V(Q_{4}):x_{4}=1\}$.}\label{fig:prop-cube-cutsets} 
\end{figure}
   
\section{Cubical polytopes} 
\label{sec:cubical-connectivity}
The aim of this section is to prove \cref{thm:cubical-connectivity},  a result that relates the connectivity of a cubical polytope to its minimum degree. 

Two vertex-edge paths are {\it independent} if they share no inner vertex. Similarly, two facet-ridge paths are {\it independent} if they do not  share an inner facet.

Given sets $A,B$ of vertices in a graph, a path from $A$ to $B$, called an {\it $A-B$ path}, is a (vertex-edge) path $L:=u_{0}\ldots u_{n}$ in the graph such that $V(L)\cap A=\{u_{0}\}$  and $V(L)\cap B=\{u_{n}\}$. We write $a-B$ path instead of $\{a\}-B$ path, and likewise, write $A-b$ path instead of $A-\{b\}$. 

Our exploration of the connectivity of cubical polytopes starts with a statement about the connectivity of the star of a vertex. But first we need a lemma that holds for all $d$-polytopes.

\begin{lemma}\label{lem:star-minus-facet-F} Let $P$ be a $d$-polytope with $d\ge 2$. Then, for any two distinct facets $F_{1}$ and $F_{2}$ of $P$, the following hold.

\begin{enumerate}[(i)]
\item There are $d$ independent facet-ridge paths between $F_{1}$ and $F_{2}$ in $P$.
\item Let  $\St$ be  the star of a vertex and let $F$ be a facet of $\St$. If $F_{1}$ and $F_{2}$  are in $\St$ and are both different from $F$, then there exists a $(d-1,d-2)$-path between $F_{1}$ and $F_{2}$ in $\St$ that does not contain $F$. 

\item Let $F$ be a facet of $P$ other than $F_{1}$ and $F_{2}$. Then there exists a $(d-1,d-2)$-path between $F_{1}$ and $F_{2}$ in $P$ that does not contain $F$. 

\item Let $R$ be an arbitrary ridge of $P$. Then there exists a facet-ridge path $J_{1}\ldots J_{m}$ with $J_{1}=F_{1}$ and $J_{m}=F_{2}$ in $P$ such that $J_{\ell}\cap J_{\ell+1}\ne R$ for each $\ell\in[1,m-1]$. 
\end{enumerate}	
\end{lemma}

\begin{proof} The proof of the lemma essentially follows from dualising Balinski's theorem. 

Let $\psi$ define the natural anti-isomorphism from the face lattice of $P$ to the face lattice of its dual $P^{*}$. 

(i). Any two  independent vertex-edge paths  in $P^{*}$ between the vertices $\psi(F_{1})$ and $\psi(F_{2})$ correspond to two  independent facet-ridge paths in $P$ between the facets $F_{1}$ and $F_{2}$. By Balinski's theorem there are $d$ independent $\psi(F_{1})-\psi(F_{2})$ paths in $P^{*}$, and so the assertion follows.
  
(ii). The facets in the star $\St$ of a vertex $s$ in $\B(P)$ correspond to the vertices in the facet $\psi(s)$ in $P^{*}$ corresponding to $s$. The existence of a facet-ridge path in $\St$ between any two facets $F_{1}$ and $F_{2}$ of $\St$ amounts to the existence of a vertex-edge path in $\psi(s)$ between the vertices $\psi(F_{1})$ and $\psi(F_{2})$ of $\psi(s)$. Since the  graph of the facet $\psi(s)$ is $(d-1)$-connected (Balinski's theorem), by Menger's theorem (\cite{Menger1927}; see also \cite[Sec.~3.3]{Die05}) there are $d-1$ independent paths between $\psi(F_{1})$ and $\psi(F_{2})$. Hence we can pick one such path $L^{*}$  that avoids the vertex $\psi(F)$ of $\psi(s)$. Dualising this path $L^{*}$ gives a $(d-1,d-2)$-path $L$ between $F_{1}$ and $F_{2}$ in the star $\St$ that does not contain the facet $F$ of $P$.

(iii). By (i) there are $d$ independent facet-ridge paths between $F_{1}$ and $F_{2}$ in $P$, and since $d\ge 2$, we can pick one such path that does not contain $F$.

(iv). Again by (i), there are $d$ independent facet-ridge paths between $F_{1}$ and $F_{2}$ in $P$, and since $d\ge 2$ and the ridge $R$ can be present in at most one such path, there must exist a facet-ridge path that does not contain $R$. The assertion now follows. 
\end{proof}

For a path $L:=u_{0}\ldots u_{n}$ we write $u_{i}Lu_{j}$ for $0\le i\le j\le n$  to denote the subpath $u_{i}\ldots u_{j}$.
 
\begin{proposition}\label{prop:star-minus-facet}  Let $F$ be a facet in the star $\St$  of a vertex in a cubical $d$-polytope. Then the antistar of $F$ in $\St$ is a strongly connected $(d-2)$-complex. 
\end{proposition}

\begin{proof}  Let $s$ be a vertex of a facet $F$ in a  cubical $d$-polytope $P$ and let $F_{1},\ldots,F_{n}$ be the facets in the star $\St$ of the vertex $s$. Let $F_{1}=F$. The result is true for $d=2$: the antistar of $F$ is just a vertex, a strongly connected 0-complex. So assume $d\ge 3$.

According to \cref{lem:cube-face-complex}, the antistar  of $F_{i}\cap F_{1}$ in $F_{i}$, the subcomplex of $F_{i}$ induced by $V(F_{i})\setminus V(F_{i}\cap F_{1})$, is a strongly connected $(d-2)$-complex for each $i\in [2,n]$. Since  \[\a-st(F_{1},\St)=\bigcup_{i=2}^{n} \a-st(F_{i}\cap F_{1}, F_{i}),\] it follows  that $\a-st(F_{1},\St)$ is a pure $(d-2)$-complex. It remains to prove that there exists a $(d-2,d-3)$-path $L$ between any two ridges $R_{i}$ and $R_{j}$ in $\a-st(F_{1},\St)$.   

By virtue of \cref{lem:cube-face-complex}, we can assume that $R_{i}\in \a-st(F_{i}\cap F_{1}, F_{i})$ and $R_{j}\in \a-st(F_{j}\cap F_{1}, F_{j})$ for  $i\ne j$ and $i,j\in [2,n]$. Since $\St$ is a strongly connected $(d-1)$-complex (\cref{prop:st-ast-connected-complexes}), there exists a $(d-1,d-2)$-path $M:=J_{1}\ldots J_{m}$ in  $\St$, where $J_{\ell}\cap J_{\ell+1}$ is a ridge for $\ell\in [1,m-1]$, $J_{1}=F_{i}$ and $J_{m}=F_{j}$. Let $E_{0}:=R_{i}$ and $E_{m}:=R_{j}$.

We can assume the path $M$ doesn't contain $F_{1}$ (\cref{lem:star-minus-facet-F}(ii)). Let us show that the path \(L\) exists, by proving the following statement by induction.

\begin{claim}
If \(\ell\leq m\),
there exists a \((d-2,d-3)\) path in \(\bigcup_{i=1}^\ell \a-st(J_i\cap F_1,J_i)\) between \(E_0\in \a-st(J_1\cap F_1,J_1)\) and any ridge \(E_\ell\in \a-st(J_\ell\cap F_1,J_\ell)\).
\end{claim}

\begin{claimproof}
The statement is true for \(\ell=1\). The complex  \(\a-st(J_1\cap F_1,J_1)\) is a strongly connected $(d-2)$-complex and contains \(E_{0}\). So an induction on $\ell$ can start.

Suppose that the statement is true for some \(\ell< m\). We show the existence of a $(d-2,d-3)$-path between $E_{0}$ and any ridge of \(\a-st(J_{\ell+1}\cap F_1,J_{\ell+1})\). 

 Let $E_{\ell}$ be a ridge in $\a-st(J_{\ell}\cap F_{1}, J_{\ell})$ such that $E_{\ell}$ contains a $(d-3)$-face $I_{\ell}$ of $\a-st(J_{\ell}\cap F_{1}, J_{\ell})\cap \a-st(J_{\ell+1}\cap F_{1}, J_{\ell+1})$. By the induction hypothesis there exists a \((d-2,d-3)\) path \(L_\ell\) in \(\bigcup_{i=1}^{\ell} \a-st(J_i\cap F_1,J_i)\) between \(E_{0}\) and \(E_{\ell}\).

Consider a ridge \(E'_{\ell+1}\in \a-st(J_{\ell+1}\cap F_1,J_{\ell+1})\) such that $E'_{\ell+1}$ contains the aforementioned \((d-3)\)-face \(I_{\ell}\). There is a \((d-2,d-3)\)-path \(L'_{\ell+1}\) in \(\a-st(J_{\ell+1}\cap F_1,J_{\ell+1})\) from \(E'_{\ell+1}\) to any ridge \(E_{\ell+1}\) in \(\a-st(J_{\ell+1}\cap F_1,J_{\ell+1})\), thanks to \(\a-st(J_{\ell+1}\cap F_1,J_{\ell+1})\) being a strongly connected $(d-2)$-complex.

Since $E_{\ell}$ and $E'_{\ell+1}$ share $I_{\ell}$, a path $L_{\ell+1}$ from $E_{0}$ to the arbitrary ridge $E_{\ell+1}$ is obtained as $L_{\ell+1}=E_{0}L_{\ell}E_{\ell}E'_{\ell+1}L'_{\ell+1}E_{\ell+1}$.

For this concatenation to work it remains to prove that the complex $\a-st(J_{\ell}\cap F_{1},J_{\ell})\cap \a-st(J_{\ell+1}\cap F_{1},J_{\ell+1})$ contains the aforementioned $(d-3)$-face $I_{\ell}$. Let $K_{\ell}:=J_{\ell}\cap J_{\ell+1}\cap F_{1}$. Because  $J_{\ell}\cap J_{\ell+1}$ is a ridge but not of $F_{1}$ and because $\{s\}\subseteq V(J_{\ell})\cap V(J_{\ell+1})\cap V(F_{1})$, we find that $0\le \dim K_{\ell}\le d-3$. From $J_{\ell}\cap J_{\ell+1}$ being a $(d-2)$-cube and $\dim K_{\ell}\le d-3$  follows the existence of a $(d-3)$-face in $ J_{\ell}\cap J_{\ell+1}$ that is disjoint from $F_{1}$, our $I_{\ell}$. As a consequence, this face $I_{\ell}\in \a-st(J_{\ell}\cap F_{1},J_{\ell})\cap \a-st(J_{\ell+1}\cap F_{1},J_{\ell+1})$, as desired.   
\end{claimproof} 

Applying the claim to \(\ell=m\) gives the existence of a path in \(\cup_{i=1}^m \a-st(J_i\cap F_1,J_i)\) between $E_{0}=R_{i}$ and $E_{m}=R_{j}$ ; this is the desired path $L$.\end{proof}
 
The proof method used in \cref{prop:star-minus-facet}  also proves the following.

\begin{theorem}\label{thm:polytope-minus-facet}  Let $F$ be a proper face of a  cubical $d$-polytope $P$. Then the antistar of $F$ in $P$ contains a spanning strongly connected $(d-2)$-subcomplex. 
\end{theorem}  

\begin{proof} Let $F_{1},\ldots,F_{n}$ be the facets of $P$ and let $F$ be a proper face of $P$. The result is true for $d=2$: the antistar of $F$ is a strongly connected 1-complex, and thus, contains a spanning 0-complex. So assume $d\ge 3$. 

Let \[\C_{r}:=\B(F_{r})-V(F).\]
If $F_{r}=F$ then $\C_{r}=\emptyset$, and if $F_{r}\cap F=\emptyset$ then $\C_{r}$ is the boundary complex of $F_{r}$, a strongly connected $(d-2)$-subcomplex of $F_{r}$ (\cref{prop:st-ast-connected-complexes}). Otherwise, $\C_{r}$ is the antistar  of $F_{r}\cap F$ in $F_{r}$, also a  strongly connected $(d-2)$-subcomplex of $F_{r}$ (\cref{lem:cube-face-complex}). 

Let \[\C:=\bigcup_{r=1}^{n} \C_{r}.\]  We show that $\C$ is the required spanning strongly connected $(d-2)$-subcomplex of $P-V(F)$, the antistar of $F$ in $P$. It follows that $\C$ is a spanning pure $(d-2)$-subcomplex of $P-V(F)$. It remains to prove that there exists a $(d-2,d-3)$-path $L$ in $\C$ between any two ridges $R_{i}$ and $R_{j}$ of $\C$ with $i\neq j$.   
 
If $R_{i},R_{j}\in \C_{r}$ for some $r\in [1,n]$, then, since $\C_{r}$ is a strongly connected $(d-2)$-complex (\cref{lem:cube-face-complex}), there exists a $(d-2,d-3)$-path in $\C_{r}$ between the two ridges $R_{i}$ and $R_{j}$. Therefore, we can assume that $R_{i}$ is in $\C_{i}$ and $R_{j}$ is in $\C_{j}$  for $i\ne j$. Observe that $F_{i}\ne F$ and $F_{j}\ne F$.  Hereafter we let $E_{0}:=R_{i}$ and $E_{m}:=R_{j}$.

Since $\B(P)$ is a strongly connected $(d-1)$-subcomplex of $P$, there exists a $(d-1,d-2)$-path $M:=J_{1}\ldots J_{m}$ in $P$ where  $J_{\ell}\cap J_{\ell+1}$ is a ridge for $\ell\in [1,m-1]$, $J_{1}=F_{i}$ and $J_{m}=F_{j}$. Each facet $J_{r}$ coincides with a facet $F_{i_{r}}$ for some $i_{r}\in[1,n]$; we henceforth  let $\D_{r}:=\C_{i_{r}}$.

By \cref{lem:star-minus-facet-F}(iii)-(iv) we can assume  that $J_{r}\ne F$ for $r\in[1,m]$ in the case of $F$ being a facet and that $J_{\ell}\cap J_{\ell+1}\ne F$ for $\ell\in[1,m-1]$ in the case of $F$ being a ridge. As a consequence, $\dim (J_{\ell}\cap J_{\ell+1}\cap F)\le d-3$; this in turn implies that, for each $\ell\in[1,m-1]$, $J_{\ell}\cap J_{\ell+1}$ contains a $(d-3)$-face $I_{\ell}$ that is disjoint from $F$. Hence $I_{\ell}\in \D_{\ell}\cap\D_{\ell+1}$ for each $\ell\in[1,m-1]$.
 
As in the proof or Proposition~\ref{prop:star-minus-facet}, we show that the path $L$ exists by proving the following claim by induction.

\begin{claim}
If \(\ell\leq m\),
there exists a \((d-2,d-3)\) path in \(\bigcup_{i=1}^\ell \D_{i}\) between \(E_0\in \D_{1}\) and any ridge $E_{\ell}\in \D_{\ell}\).
\end{claim}
\begin{claimproof}
The statement is true for \(\ell=1\). The complex  \(\D_{1}\) is a strongly connected $(d-2)$-complex and \(E_0\in \D_{1}\).

Suppose that the statement is true for some \(\ell< m\). We show the existence of a $(d-2,d-3)$-path between $E_{0}\in \D_{1}$ and any ridge of $\D_{\ell+1}$. 

 Let \(E_{\ell}\) be a ridge in \(\D_{\ell}\) containing a \((d-3)\) face \(I_{\ell}\) of \(\D_{\ell} \cap \D_{\ell+1}\); this $(d-3)$-face $I_{\ell}$ exists by our previous discussion. By the induction hypothesis, there exists a \((d-2,d-3)\) path \(L_\ell\) in \(\bigcup_{i=1}^{\ell} \D_i\) between \(E_0\) and the ridge \(E_{\ell}\).

Consider a ridge \(E'_{\ell+1}\in \D_{\ell+1}\) containing the face \(I_{\ell}\). There is a $(d-2,d-3)$-path $L'_{\ell+1}$ in \(\D_{\ell+1}\) from \(E'_{\ell+1}\) to any ridge \(E_{\ell+1}\in \D_{\ell+1}\), thanks to $\D_{\ell+1}$ being a strongly connected $(d-2)$-complex. 

The desired path $L_{\ell+1}$ between $E_{0}$ and the arbitrary ridge $E_{\ell+1}$ is obtained as $L_{\ell+1}=E_{0}L_{\ell}E_{\ell}E'_{\ell+1}L_{\ell+1}E_{\ell+1}$.
\end{claimproof}

The claim for \(\ell=m\) gives the desired \((d-2,d-3)\)-path $L$ in $\cup_{i=1}^{m}\D_{i}\subset \C$ between $E_{0}=R_{i}$ and $E_{m}=R_{j}$, which concludes the proof.
\end{proof}

\begin{remark}\label{rmk:Antistar-Cubical-nonpure}  \cref{thm:polytope-minus-facet} is best possible in the sense that the antistar of a face does not always contain a spanning strongly connected $(d-1)$-subcomplex. The removal of the vertices of the face $F$ in \cref{fig:cubical-3-polytopes} leaves a pure $(d-1)$-subcomplex that is not strongly connected.
\end{remark}

The ideas presented in \cref{prop:star-minus-facet,thm:polytope-minus-facet} play a key role in the proof of the main result of \cite{ThiPinUgo18}. 

Before proving the main result of the section, we state a useful corollary that follows from \cref{prop:connected-complex-connectivity,thm:polytope-minus-facet}.

\begin{corollary}\label{cor:Removing-facet-(d-2)-connectivity}  Let $P$ be a cubical $d$-polytope and let $F$ be a proper face of $P$. Then the subgraph  $G(P)-V(F)$ is $(d-2)$-connected.
\end{corollary} 
 
For $d\ge 4$ we define the two functions $f(d)$ and $g(d)$ that we mentioned in the introduction.

\begin{enumerate}
\item 	The function $f(d)$ gives the maximum number such that every cubical $d$-polytope with minimum degree $\delta\le f(d)$ is $\delta$-connected. 
\item the function $g(d)$ gives the maximum number such that every minimum separator with cardinality at most $g(d)$ of every cubical $d$-polytope consists of the neighbourhood of some vertex.
\end{enumerate}

The functions $f(3)$ and $g(3)$ are not defined. No cubical 3-polytope has minimum degree $\delta\ge 4$, and so for every positive integer $\delta_{0}\ge 3$ it follows that every cubical 3-polytope with minimum degree $\delta\le \delta_{0}$ is $\delta$-connected. \Cref{fig:cubical-3-polytopes} shows cubical 3-polytopes with minimum separators that are not the neighbourhood of a vertex.  

The function $f(d)$ is well defined for $d\ge 4$. There is a cubical $d$-polytope with minimum degree $\delta$ for every $\delta\ge d\ge 4$, for instance, a neighbourly cubical $d$-polytope \cite{JosZie00}. Every $d$-polytope is $d$-connected by Balinski's theorem. Furthermore, there exists a cubical $d$-polytope with minimum degree $\delta>2^{d-1}$ that is not $\delta$-connected: the connected sum of two copies of a neighbourly cubical $d$-polytope with minimum degree $\delta$.  Thus $d\le f(d)\le 2^{d-1}$.  

At this moment, we don't claim that $g(d)$ exists; this will become evident in the proof of \cref{thm:cubical-connectivity}.

 \begin{proposition}\label{prop:g+1} Let $P$ be a cubical $d$-polytope with $d\ge 4$. If the function $g(d)$ exists and $P$ has minimum degree at least $g(d)+1$, then $G(P)$ is $(g(d)+1)$-connected.
\end{proposition}
	
\begin{proof} Suppose that $G(P)$ is not $(g(d)+1)$-connected. Then there is a minimum separator $X$ with cardinality at most $g(d)$. By the definition of $g(d)$, $X$ consists of all the neighbours of some vertex $u$. This contradicts the degree of $u$, which is at least $g(d)+1>|X|$. 
\end{proof}

\begin{corollary}\label{cor:f-g} If the function $g(d)$ exists for $d\ge 4$, then $f(d)>g(d)$.
\end{corollary}

\begin{theorem}[Connectivity Theorem]\label{thm:cubical-connectivity} A cubical $d$-polytope $P$ with minimum degree $\delta$ is $\min\{\delta,2d-2\}$-connected for every $d\ge 3$. 

Furthermore, for any $d\ge 4$, every minimum separator $X$ of cardinality at most $2d-3$ consists of  all the neighbours of some vertex, and the subgraph $G(P)-X$ contains  exactly two components, with one of them being the vertex itself.  \end{theorem}
 
\begin{proof} Let $0\le \alpha\le d-3$ and let $P$ be a cubical $d$-polytope with minimum degree at least $d+\alpha$.  Let $G:=G(P)$.
 
We first prove that $P$ is $(d+\alpha)$-connected. The case of $d=3$ follows from Balinski's theorem. So assume $d\ge 4$.  Let $X$ be a minimum separator of $P$. Throughout the proof, let $u$ and $v$ be two distinct vertices that belong to $G-X$  and are disconnected by $X$. The theorem follows from a number of claims that we prove next.

\begin{claim}\label{cl:d-1} If $|X|\le d+\alpha$ then, for any facet $F$, the cardinality of $X\cap V(F)$ is at most $d-1$.
\end{claim}  
\begin{claimproof}  
Suppose otherwise and let $F$ be a facet with $|X\cap V(F)|\ge d$. 
Let \[G':=G-V(F).\] According to \cref{cor:Removing-facet-(d-2)-connectivity}, the subgraph $G'$ is $(d-2)$-connected. Since there are at most $\alpha\le d-3$ vertices in $V(G')\cap X$, removing from $G'$ the vertices in $V(G')\cap X$ doesn't disconnect $G'$. 

We show there is a $u-v$ path in $G- X$, which would be a contradiction and prove the claim. If $u,v\in V(G')\setminus X$ then there is a $u-v$ path in $G'-X$, as $G'-X$ is connected. So assume $u\in V(F)\setminus X$. Since $u$ has degree at least $d+\alpha$ and since every vertex in $F$ has at least $d+\alpha-(d-1)=\alpha+1$ neighbours outside $F$ (in $G'$), at least one of them, say $u_{G'}$, is in $V(G') \setminus X$. Likewise either $v\in V(F)\setminus X$  and  there is a neighbour $v_{G'}$ of $v$ in $V(G')\setminus X$ or $v\in G'-X$. Therefore, if $v\in V(F)\setminus X$ then there is a $u-v$ path $L$ in $G- X$ that contains a subpath $L'$ in $G'$ between the vertices $u_{G'}$ and $v_{G'}$ in $V(G') \setminus X$; that is, $L=uu_{G'}L'v_{G'}v$. If instead $v\in G'- X$ then there is a $u-v$ path $L$ in $G- X$ passing through the vertex $u_{G'}$ and containing a subpath $L':=u_{G'}-v$ in $G'-X$; that is, $L=uu_{G'}L'v$. Hence there is always a $u-v$ path in $G-X$, and thus, $G$ is not disconnected by $X$, a contradiction.   
\end{claimproof} 

\begin{claim}\label{cl:d-facets} If $|X|\le d+\alpha$, then there exist facets $F_{1},\ldots,F_{d}$ of $P$ such that $G(F_{i})$ is disconnected by $X$ for each $i\in[1,d]$.
\end{claim}

\begin{claimproof} 
Suppose by way of contradiction that $X$ disconnects the graphs of at most $k$ facets $F_{1},\ldots, F_{k}$ of $P$ with $k\le d-1$.  We find a $u-v$ path in $G-X$, which would contradict $X$ being a separator of $G$.

There are at least $d$ facets containing $u$ and there are at least $d$ facets containing $v$. As a result, we can pick facets $K_{u}$ and $K_{v}$ with $u\in K_{u}$ and $v\in K_{v}$ whose graphs are not disconnected by $X$; that is $K_{u},K_{v}\not\in\{F_{1},\ldots,F_{k}\}$. If $K_{u}=K_{v}$ then we can find a $u-v$ path in  $G(K_{u})- X$. So assume $K_{u}\ne K_{v}$. Since $\B(P)$ is a strongly connected $(d-1)$-complex and since there are at least $d$ independent $(d-1,d-2)$-paths from $K_{u}$ to $K_{v}$ in $\B(P)$ (\cref{lem:star-minus-facet-F}(i)), there exists a $(d-1,d-2)$-path $J_{1}\ldots J_{n}$ in $\B(P)$ with $J_{1}=K_{u}$ and $J_{n}=K_{v}$ such that $\{J_{1},\ldots,J_{n}\}\cap \{F_{1},\ldots,F_{k}\}=\emptyset$. As a consequence, the subgraphs $G(J_{i})$ are not disconnected by $X$. 

Construct a $u-v$ path  $L$ by traversing the facets $J_{1},\ldots, J_{n}$ as follows: find a path $L_{1}$ in $J_{1}$ from $u$ to a vertex in $J_{1}\cap J_{2}$, then a path $L_{2}$ in $J_{2}$ from $J_{1}\cap J_{2}$ to $J_{2}\cap J_{3}$ and so on up to a path $L_{n-1}$ in $J_{n-1}$ from $J_{n-2}\cap J_{n-1}$ to $J_{n-1}\cap J_{n}$; here use the connectivity of the subgraphs $G(J_{1})- X,\ldots, G(J_{n-1})- X$. Finally, find a path $L_{n}$ in $J_{n}=K_{v}$ from $J_{n-1}\cap J_{n}$ to the vertex $v$ using the connectivity of $G(J_{n})- X$. The path $L$ is the concatenation of the paths $L_{1},\ldots,L_{n}$. 

The aforementioned concatenation works as long as there is at least one vertex in $V(J_{\ell}\cap J_{\ell+1})\setminus X$ for each $\ell\in [1,n-1]$. For $d\ge 4$, it follows that $|V(J_{\ell}\cap J_{\ell+1})|=2^{d-2}\ge d$, which  is greater that $|V(J_{\ell})\cap X|\le d-1$ by \cref{cl:d-1}. Hence $V(J_{\ell}\cap J_{\ell+1})\setminus X$ is nonempty, and consequently, the $u-v$ path $L$ always exists and completes the proof of the claim. 
\end{claimproof}

\begin{claim}\label{cl:connectivity} If $|X|\le d+\alpha$ then $|X|=d+\alpha$. 
\end{claim}

\begin{claimproof} Let $F$ be a facet of $P$ whose graph is disconnected by $X$, which by \cref{cl:d-facets} exists. \cref{cl:d-1} together with Balinski's theorem ensures that $|V(F)\cap X|=d-1$. Let $G':=G-V(F)$.  By \cref{cor:Removing-facet-(d-2)-connectivity}, $G'$ is  a $(d-2)$-connected subgraph of $G$. 

Suppose that a minimum separator $X$ has size at most $d-1+\alpha$; we show that $X$ does not disconnect $G$ by finding a $u-v$ path $L$ between the vertices $u$ and $v$ of $G-X$, which would be a contradiction. 

There are at most $\alpha\le d-3$ vertices in $V(G')\cap X$, and so removing $V(G')\cap X$ from $G'$ doesn't disconnect $G'$. 

If $u$ and $v$ are both in $G'$ then there is a $u-v$ path in $G'$ that is disjoint from $X$. So assume that $u\in V(F)\setminus X$. Let $X_{1}$ denote the set of neighbours of $u$ in $G'$; then $|X_{1}|\ge \alpha+1$, since $u$ has at least $d+\alpha$ neighbours in $P$, with exactly $d-1$ of them in $F$. As a consequence, there is a neighbour $u_{G'}$ of $u$ in $V(G') \setminus X$. Likewise either $v\in V(F)\setminus X$ and there is a neighbour $v_{G'}$ of $v$ in $V(G')\setminus X$ or $v\in G'-X$. If $v\in V(F) \setminus X$, there is a $u-v$ path in $G- X$ that passes through the vertices $u_{G'}$ and $v_{G'}$ of $V(G') \setminus X$. If instead $v\in G'-X$, there is a $u-v$ path $L$ in $G- X$ that includes a subpath $L'$ in $G'-X$ between  $u_{G'}$ and $v$ so that $L=uu_{G'}L'v$. Hence we always have a $u-v$ path in $G-X$. This contradiction shows that a minimum separator has size {\it exactly} $d+\alpha$.
\end{claimproof}

{\bf From \cref{cl:connectivity} it follows that $P$ is $(d+\alpha)$-connected.} The structure of a minimum separator is settled in \cref{cl:separator-d-4,cl:separator-5-d-3}. For every $d\ge 4$, \cref{cl:separator-d-4} settles  the case $\alpha\le d-4$ and \cref{cl:separator-5-d-3}  the case $\alpha= d-3$. 

\begin{claim}\label{cl:separator-d-4} If $\alpha\le d-4$, then the set $X$ consists of the neighbours of some vertex and the minimum degree of $P$ is exactly $d+\alpha$.
\end{claim}
\begin{claimproof}
As in \cref{cl:connectivity}, let $F$ be a facet of $P$ whose graph is disconnected by $X$ and let $G':=G-V(F)$. Then $G'$ is a $(d-2)$-connected subgraph of $G$ (\cref{cor:Removing-facet-(d-2)-connectivity}). Besides, $|V(F)\cap X|=d-1$ by a combination of \cref{cl:d-1} and Balinski's theorem.

Since there are  {\it exactly} $\alpha+1\le d-3$ vertices in $V(G')\cap X$, removing $V(G')\cap X$ from $G'$ doesn't disconnect $G'$. We may therefore assume that $u\in V(F)\setminus X$.

If there is a path $L_{u}$  in $G-X$ from $u$ to a vertex $u_{G'}\in G'-X$ and a path $L_{v}$  in $G-X$ from $v$ to a vertex $v_{G'}$ in $G'-X$ so that $L_{u}$ and $L_{v}$ are both disjoint from $X$, then we get a $u-v$ path $L$ in $G- X$ defined as $L=uL_{u}u_{G'}L'v_{G'}L_{v}v$ where $L'$ is a path in $G'-X$ between $u_{G'}$ and $v_{G'}$. Recall the minimum degree of $u$ is at least $d+\alpha$.
 
We  may therefore assume that $u$ is in $V(F)\setminus X$ and that there is no such path $L_{u}$ in $G-X$  from $u$ to $G'-X$.  The set $X_{1}$ of neighbours of $u$ in $G'$ must then be a subset of $X$, and since $|X_{1}|\ge \alpha+1$, it follows that $X_{1}=V(G')\cap X$, and thus, that $|X_{1}|=\alpha+1$. In addition, every path of length two from $u$ to $G'$ passing through a neighbour of $u$ in $F$ contains some vertex from $X$; otherwise  the aforementioned path $L_{u}$ would exist. Let $X_{2}$ denote the vertices in $X$ that are present in a $u-V(G')$ path of length two passing through a neighbour of $u$ in $F$. Every vertex of $F$ has a neighbour in $G'$, and so there is a $u-V(G')$ path through each neighbour of $u$ in $F$, $d-1$ such neighbours in total. Since there are no triangles in $P$, we get $X_{1}\cap X_{2}=\emptyset$, which in turn implies that $X_{2}\subset V(F)$. Hence $|X_{2}|=d-1$, and every neighbour of $u$ in $F$ is in $X$.  Consequently, the degree of  $u$, $|X_{1}|+|X_{2}|$, is precisely $d+\alpha$, and the set $X$ consists of the $d+\alpha$ neighbours of $u$ in $P$,  as desired.
\end{claimproof}

\begin{claim}\label{cl:separator-5-d-3} If $\alpha=d-3$, then the set $X$ consists of the neighbours of some vertex and the minimum degree of $P$ is exactly $d+\alpha$.
\end{claim}

\begin{claimproof}  Proceed by contradiction: every vertex in $P$ has at least one neighbour outside $X$. 

By \cref{cl:d-1} there are at most $d-1$ vertices from $X$ in any facet $F$ of $P$. If the removal of $X$ disconnects the graph of a facet $F$, then there would be exactly $d-1$ vertices in $V(F)\cap X$, which constitute the neighbours in $F$ of some vertex  of $F$ (\cref{prop:known-cube-cutsets}). Consequently, the subgraph $G(F)-X$ would have exactly two components: one being a singleton $z(F)$ and another $Z(F)$ being $(d-3)$-connected by \cref{prop:cube-cutsets}; if $X$ doesn't disconnect $F$, we let $z(F)=\emptyset$ and let $Z(F):=G(F)-X$. Hence, for every facet $F$ of $P$, the subgraph $Z(F)$ is connected, and $V(F)=z(F)\cup V(Z(F))\cup (V(F)\cap X)$. Abusing terminology, if $z(F)\neq\emptyset$ we make no distinction between the set and its unique element.   

Since $u$ and $v$ are separated by $X$, every $u-v$ path in $G$ contains a vertex from $X$. Because the vertex $u$ has a neighbour $w$ not in $X$, there must exist a facet $F_{u}$ in which $u\in Z(F_{u})$: a facet containing the edge $uw$. Similarly, there exists a facet $F_{v}$ containing $v$ in  which $v\in Z(F_{v})$.     

Consider an arbitrary $(d-1,d-2)$-path $J_{1}\ldots J_{n}$ in $P$ with $J_{1}=F_{u}$ and $J_{n}=F_{v}$. If, for each $i\in [1,n-1]$, there is a vertex $y_{i}\in V(J_{i}\cap J_{i+1})$ with $y_{i}\in V(Z(J_{i}))\cap V(Z(J_{i+1}))$, then there would be a $u-v$ path $L$ in $G-X$ and the claim would hold.  Indeed, let $y_0:= u$ and $y_{n}: = v$. For all $i\in [0,n-1]$, there would  be a path $L_{i+1}$ in $Z(J_{i+1})$ from $y_i$ to $y_{i+1}$. Concatenating all these paths $L_{1},\ldots, L_{n}$, we would then have a $u-v$ path $L$ in $G-X$, giving a contradiction and settling the claim. We would say that a facet-ridge path $J_{1}\ldots J_{n}$ from $F_{u}$ to $F_{v}$ is  {\it valid} if the aforementioned vertex $y_{i}$ exists for each $i\in [1,n-1]$; otherwise it is {\it invalid}.
  
Hence {\bf it remains to show that, for some facet-ridge path $J_{1}\ldots J_{n}$ from $J_{1}=F_{u}$ to $J_{n}=F_{v}$,  there exists a vertex in $V(Z(J_{i}))\cap V(Z(J_{i+1}))$ for each $i\in [1,n-1]$ when $d\ge 4$.} In other words, {\bf it remains to show there exists a valid fact-ridge path from $F_{u}$ to $F_{v}$.} 

Take a facet-ridge  path $J_{1}\ldots J_{n}$ from $F_{u}$ to $F_{v}$ and suppose it is invalid; that is, $V(Z(J_{i}))\cap V(Z(J_{i+1}))=\emptyset$ for some $i\in [1,n-1]$. Then $V(Z(J_{i}))\cap V(J_{i+1})\subset z(J_{i+1})$. Therefore,
\begin{equation}
\begin{aligned}\label{eq:Claim4-1}  
	V(J_{i}\cap J_{i+1})=V(J_{i})\cap V(J_{i+1})&=\left[z(J_{i})\cup V(Z(J_{i}))\cup (V(J_{i})\cap X)\right]\cap V(J_{i+1})\\ 
	&\quad\subset z(J_{i})\cup z(J_{i+1})\cup (X\cap V(J_{i}\cap J_{i+1})).
\end{aligned} 
\end{equation}
If neither $G(J_{i})$ nor $G(J_{i+1})$ is disconnected by $X$, then $z(J_{i})=z(J_{i+1})=\emptyset$, and  by \cref{eq:Claim4-1} and \cref{cl:d-1}, \begin{equation}\label{eq:Claim4-2}2^{d-2}=|V(J_{i}\cap J_{i+1})|\le |X\cap V(J_{i})|\le d-1.\end{equation}  
If instead $G(J_{i})$ is disconnected by $X$, then $X\cap V(J_{i})$ consists of all the $d-1$ neighbours of $z(J_{i})$  in $J_{i}$ (\cref{prop:known-cube-cutsets}), and thus, $|X\cap V(J_{i}\cap J_{i+1})|\le d-2$. In this case, by \cref{eq:Claim4-1},
\begin{equation}\label{eq:Claim4-3}2^{d-2}=|V(J_{i}\cap J_{i+1})|\le 2+d-2=d.\end{equation} 
 
 \Cref{eq:Claim4-2} does not hold for $d\ge 4$, while \cref{eq:Claim4-3} only holds  for $d=4$, in which case it holds with equality.  As a consequence, if $d\ge 5$, every facet-ridge path from $F_{u}$ to $F_{v}$ is valid. As a result,  the aforementioned $u-v$ path $L$ in $G-X$ always exists for $d\ge 5$, a contradiction. {\bf This completes the case $d\ge 5$}.
  
 The case $d=4$ requires more work. Let \[X:=\{x_{1},\ldots,x_{5}\}.\] Suppose by way of contradiction that every facet-ridge path from $F_{u}$ to $F_{v}$ is invalid. Consider a particular such path $M:=J_{1}\ldots J_{n}$. Then $V(Z(J_{i}))\cap V(Z(J_{i+1}))=\emptyset$ for some $i\in [1,n-1]$, and for that index $i$, \cref{eq:Claim4-3} must hold with equality, which implies that \cref{eq:Claim4-1} must also hold with equality. Consequently, the following setting ensues. 

\begin{enumerate}
\item $|z(J_{i})\cup z(J_{i+1})|=2$; that is, $z(J_{i})\ne z(J_{i+1})$;
\item 
the graphs of the facets $J_{i}$ and $J_{i+1}$ are both disconnected by $X$;

\item the neighbours of $z(J_{i})$ in $J_{i}$ and of $z(J_{i+1})$ in $J_{i+1}$ are all from $X$; 
\item the ridge $R_{i}:=J_{i}\cap J_{i+1}$ consists of four vertices---namely, $z(J_{i})$, $z(J_{i+1})$ and two vertices from $X$, say $x_{1}$ and $x_{2}$;
\item each vertex $z(J_{i})$ and $z(J_{i+1})$ has a neighbour in $J_{i}\setminus J_{i+1}$ and $J_{i+1}\setminus J_{i}$, respectively; and
\item there is a vertex from $X$, say $x_{5}$, lying outside $J_{i}\cup J_{i+1}$. 
\end{enumerate}
Any pair of facets in this setting are said to be in {\it Configuration A} and the ridge in which they intersect is said to be {\it problematic}. For instance, the pair $(J_{i}, J_{i+1})$ is in Configuration A and the ridge $R_{i}$ is problematic; see \cref{fig:Aux-Conn-Thm-A}(a).  

For a facet-ridge path from $F_{u}$ to $F_{v}$ to be invalid, it must have a pair of facets in Configuration A.

 \begin{figure}
\includegraphics{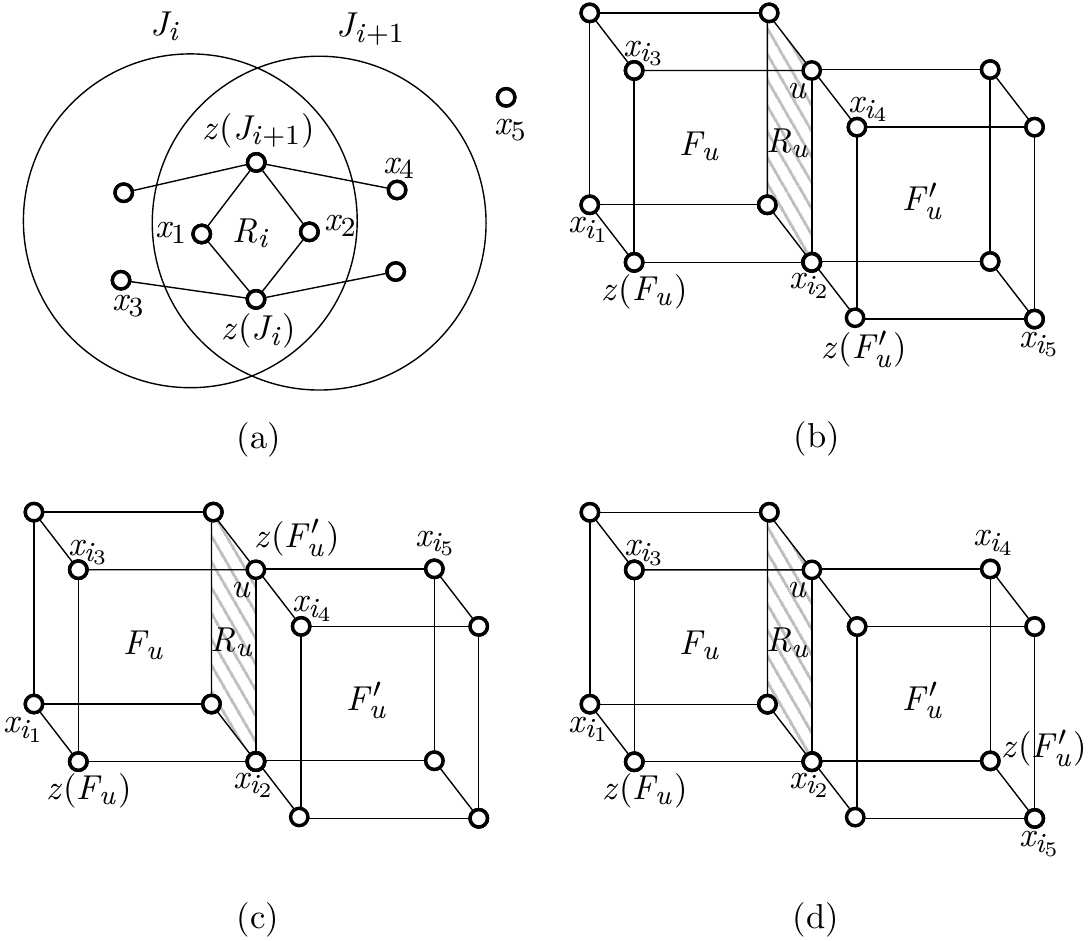}
\caption{Auxiliary figure for \cref{cl:separator-5-d-3} of \cref{thm:cubical-connectivity}. (a) Configuration A: two facets $J_{i}$ and $J_{i+1}$ whose graphs are disconnected by $X=\{x_{1},x_{2},x_{3},x_{4},x_{5}\}$ and a problematic ridge $R_{i}:=J_{i}\cap J_{i+1}$. (b)--(d) The facets $F_{u}$ and $F_{u}$ are both disconnected by $X$ and intersects at an edge. The ridge $R_{u}:=F_{u}\cap F_{u}''$ that defines a new facet $F_{u}''$ is highlighted.}\label{fig:Aux-Conn-Thm-A} 
\end{figure}

We want to be more careful when selecting the facets $F_{u}$ and $F_{v}$ and when selecting the facet-ridge path $M$ from $F_{u}$ to $F_{v}$. We require the following.
 \begin{equation}\label{claim-Fu-Fv}
\parbox{0.8\textwidth}{The facets $F_{u}$ and $F_{v}$ can be  picked so that their graphs are not disconnected by $X$; that is, $G(F_{u})-X$ and $G(F_{v})-X$ are both connected subgraphs of $G(F_{u})$ and $G(F_{v})$, respectively.} \tag{*} 
\end{equation} 
\begin{proof}[Proof of \eqref{claim-Fu-Fv}]  Suppose that the facet $F_{u}$ cannot be picked as desired. Then the graph of $F_{u}$ is disconnected by $X$, and by \cref{prop:known-cube-cutsets} there is a vertex $z(F_{u})\in G(F_{u})$ whose neighbours in $F_{u}$ are all from $X$. Say $X\cap V(F_{u})=\{x_{i_{1}},x_{i_{2}},x_{i_{3}}\}$. Recall that $u\in Z(F_{u})$. 
 
Since $u$ has degree at least five (it is nonsimple), it follows that there is a facet $F_{u}'$ in $P$ containing $u$ and intersecting $F_{u}$ at a vertex or an edge. Since $F_{u}'$ contains $u$, its graph must be disconnected by $X$ (otherwise it is the desired facet). Therefore $|X\cap V(F_{u}')|=3$, and thus, $X\cap V(F_{u})\cap V(F_{u}')\neq \emptyset$. As a consequence, we find that $F_{u}\cap F'_{u}$ is an edge between a vertex of $X$, say $x_{i_{2}}$, and $u$. It follows that $X\cap V(F'_{u})=\{x_{i_{2}},x_{i_{4}},x_{i_{5}}\}$. Three configurations are possible: \cref{fig:Aux-Conn-Thm-A}(b)--(d).

The argument remains unchanged in all the three configurations. Refer to \cref{fig:Aux-Conn-Thm-A}(b) for concreteness.  Consider the ridge $R_{u}$ of $F_{u}$ that contains the edge $ux_{i_{2}}$ but does not contain the vertex $x_{i_{3}}$; the ridge $R_{u}$ is highlighted in \cref{fig:Aux-Conn-Thm-A}(b). Let $F_{u}''$ be the facet of $P$ that intersects $F_{u}$ at $R_{u}$. Then $X\cap V(F_{u}'') \subseteq \{x_{i_{2}},x_{i_{4}}\}$, since $F_{u}\cap F_{u}''$ and $F_{u}'\cap F_{u}''$  are faces that contain $u$. Therefore, the graph of $F_{u}''$ is not disconnected by $X$ and $F_{u}''$ could have been chosen as $F_{u}$. As a consequence of this contradiction, the facet $F_{u}$ can be picked as desired.
 
Similar analysis shows that the facet $F_{v}$ can also be picked so that $G(F_{v})$ is not disconnected by $X$. This completes the proof of \eqref{claim-Fu-Fv}.
\end{proof} 

We are now ready to complete the proof of the claim by showing that we can always find a valid facet-ridge $F_{u}-F_{v}$ path. The existence of such a path would complete the proof of the claim.

There are at least four independent facet-ridge paths from $F_{u}$ to $F_{v}$ (\cref{lem:star-minus-facet-F}(i))---say $M_{a}, M_{b}, M_{c}$ and $M_{d}$---and at least four pairs of facets exhibiting  Configuration A---one per path. Each pair of facets in Configuration A  gives rise to a problematic ridge. We may assume that $M=M_{a}$. The ensuing four points are key. 
 
\begin{enumerate}
\item The facet $F_{u}$ or $F_{v}$ does not appear in any Configuration A (by Statement~\eqref{claim-Fu-Fv}). 
\item Any facet of $P$ other than $F_{u}$ and $F_{v}$ may appear in at most one facet-ridge $F_{u}-F_{v}$ path; in particular, it appears in at most one pair exhibiting Configuration A.
\item The problematic ridges are pairwise distinct, as the paths $M_{a},M_{b},M_{c},M_{d}$ as independent. 
\item Each problematic ridge appears  in precisely one of the paths $M_{a},M_{b},M_{c},M_{d}$.
\item Each nonproblematic ridge $R$ of $P$ present in a Configuration A appears in at most two paths in $\{M_{a},M_{b},M_{c},M_{d}\}$. This is so because $R$ is the intersection of two facets $F$ and $F'$, and the facet $F$ or $F'$ appears in at most one such path. 
\end{enumerate}

With a counting argument we show that Configuration A cannot occur in all the four paths. We count the ridges that contain two vertices from $X$ and are present in a Configuration A.

For every pair of facets $(J,J')$ exhibiting Configuration A, there are five ridges in $J\cup J'$ containing two vertices from $X$. For instance, for the pair $(J_{i}, J_{i+1})$ of \cref{fig:Aux-Conn-Thm-A}(a), the pairs $(x_{1},x_{2})$, $(x_{1},x_{3})$, $(x_{2},x_{3})$, $(x_{1},x_{4})$ and $(x_{2},x_{4})$ induce the five ridges. So, considering the four aforementioned $F_{u}-F_{v}$ paths, we have a total of twenty ridges that are present in a Configuration A and contain two vertices from $X$. Besides, there are ten ways of pairing two vertices from $X$, and thereby there are at most ten distinct ridges containing two vertices from $X$. 

Each problematic ridge appears in precisely one Configuration A; there are at least four problematic ridges, and therefore, there are at most six nonproblematic ridges. Each nonproblematic ridge that contains two vertices from $X$ appears in two facets, and consequently in at most two pairs of facets exhibiting Configuration A (that is,  in at most two Configurations A).  We account for the ten ridges containing two vertices from $X$ in the four Configurations A: at least four problematic ridges and at most six  nonproblematic ones. This means that we can have at most $4+6\times 2=16$ ridges that contain two vertices from $X$ and appear in the four Configurations A. Since the four Configurations A require twenty ridges containing two vertices from $X$, we can choose a $(d-1,d-2)$-path  $F_{u}-F_{v}$  in which Configuration A doesn't occur. {\bf This completes the case $d=4$}, and with it the proof of the claim.
\end{claimproof} 
 
We now complete the proof of the theorem. \cref{cl:separator-d-4,cl:separator-5-d-3} ensure that, for $d\ge 4$, a minimum separator $X$ with cardinality at most $2d-3$ in a cubical $d$-polytope consists of the neighbours of a vertex. Thus {\bf the function $g(d)$ exists and satisfies $2d-3\le g(d)$.} 

From \cref{cor:f-g} it then follows that $f(d)\ge 2d-2$; in other words, a cubical $d$-polytope with minimum degree $\delta$ is $\min\{\delta,2d-2\}$-connected. This completes the proof of the theorem.
\end{proof}

A simple corollary of \cref{thm:cubical-connectivity} is the following.
 
\begin{corollary}\label{cor:cubical-nsimple-connectivity}  A cubical $d$-polytope with no simple vertices is $(d+1)$-connected. 
\end{corollary}

As we mentioned in the introduction an open problem that nicely arises from \cref{thm:cubical-connectivity} is \cref{prob:bounds}.

\section{Acknowledgments} The authors would like to thank the anonymous referees for their very detailed comments and suggestions. The presentation of the paper has greatly benefited from their input.

%\section{Concluding remarks}

%---REFERENCES---
 
%\providecommand{\bysame}{\leavevmode\hbox to3em{\hrulefill}\thinspace}

\providecommand{\MR}{\relax\ifhmode\unskip\space\fi MR }

\providecommand{\MRhref}[2]{%
  \href{http://www.ams.org/mathscinet-getitem?mr=#1}{#2}
}
    
\providecommand{\href}[2]{#2}

\end{document}